\documentclass[a4paper, 12pt, oneside]{amsart}

\usepackage[T1]{fontenc}
\usepackage[utf8]{inputenc}
\usepackage{mathptmx}       
\usepackage{helvet}         
\usepackage{courier}        
\usepackage{type1cm}        
\usepackage{graphicx}        

\usepackage{amssymb, amsmath}
\usepackage{enumitem}
\usepackage[DIV=10]{typearea}

\newcommand{\Z}{\mathbb{Z}}
\newcommand{\N}{\mathbb{N}}
\newcommand{\D}{\mathcal{D}}
\newcommand{\G}{\mathcal{G}}
\newcommand{\C}{\mathbb{C}}

\newcommand{\reg}{{\operatorname{reg}}}
\newcommand{\sing}{{\operatorname{sing}}}

\theoremstyle{theorem}
\newtheorem{theorem}{Theorem}
\newtheorem{proposition}[theorem]{Proposition}
\newtheorem{corollary}[theorem]{Corollary}

\theoremstyle{definition}

\newtheorem{remark}[theorem]{Remark}
\newtheorem*{question}{Question}

\title{$*$-isomorphism of Leavitt path algebras over $\mathbb{Z}$}

\author{Toke Meier Carlsen}
\address{University of the Faroe Islands\\Department of Science and Technology\\
N\'oat\'un 3\\ FO-100 T\'orshavn\\ Faroe Islands}
\email{toke.carlsen@gmail.com}

\date{\today}

\begin{document}

\begin{abstract}
We characterise when the Leavitt path algebras over $\Z$ of two arbitrary countable directed graphs are $*$-isomorphic by showing that two Leavitt path algebras over $\Z$ are $*$-isomorphic if and only if the corresponding graph groupoids are isomorphic (if and only if there is a diagonal preserving isomorphism between the corresponding graph $C^*$-algebras). We also prove that any $*$-homomorphism between two Leavitt path algebras over $\Z$ maps the diagonal to the diagonal. Both results hold for more general subrings of $\C$ than just $\Z$.
\end{abstract}

\maketitle

\section{Introduction}

\emph{Leavitt path algebras} were introduced independently in \cite{AA} and \cite{AMP} as algebraic analogues of graph $C^*$-algebras and have since then attracted a lot of attention, both in connection with graph $C^*$-algebras and as interesting algebraic objects on their own (they are called Leavitt path algebras because they generalise certain algebras studied by Leavitt in \cite{L1,L2,L3}). 

The Leavitt path algebra of a directed graph $E$ over a unital commutative ring $R$ is a universal $R$-algebra $L_R(E)$ whose generators and relations are determined by $E$ (see Section \ref{lpa} for the precise definition of $L_R(E)$). Each involution (for example the identity map) of $R$ gives rise to an involution of $L_R(E)$ which is therefore a $*$-algebra. 

It is natural to ask when two Leavitt path algebras are ($*$-)isomorphic. Abrams and Tomforde showed in \cite{AT} that if two Leavitt path algebras $L_\C(E)$ and $L_\C(F)$ over $\C$ are $*$-isomorphic, then so are the corresponding graph $C^*$-algebras $C^*(E)$ and $C^*(F)$ (their proof is easily generalised to subrings of $\C$ that are closed under complex conjugation and contains $1$). 

Johansen and Sørensen gave in \cite{JS} the first example of two Leavitt path algebras that are not $*$-isomorphic in spite of the corresponding graph $C^*$-algebras being isomorphic, when they showed that the Leavitt path algebras over $\Z$ of $E_2$ and $E_{2-}$ are not $*$-isomorphic (that $C^*(E_2)$ and $C^*(E_{2-})$ are isomorphic was proved by Rørdam in \cite{Ror} as an important step towards classifying simple Cuntz-Krieger algebras). 

Each Leavitt path algebra $L_R(E)$ contains a certain abelian subalgebra $D_R(E)$ called \emph{the diagonal}. Johansen and Sørensen obtained their result by showing that when $E$ is a finite graph and $R$ is a subring of $\C$ satisfying certain conditions, then every projection in $L_R(E)$ belongs to $D_R(E)$. We generalise this result to arbitrary graphs $E$ and more general subrings $R$ of $\C$ (Proposition \ref{pro}). It follows than any $*$-homomorphism between two Leavitt path algebras $L_\Z(E)$ and $L_\Z(F)$ over $\Z$ is \emph{diagonal preserving} in the sense that it maps $D_\Z(E)$ to $D_\Z(F)$ (Corollary \ref{cor}).   

From a countable directed graph $E$, one can construct a topological groupoid $\G_E$ such that the $C^*$-algebra of $\G_E$ is isomorphic to the graph $C^*$-algebra $C^*(E)$, and the Steinberg algebra $A_R(\G_E)$ is isomorphic to $L_R(E)$. It is proved in \cite{BCW} that two graph $C^*$-algebras $C^*(E)$ and $C^*(F)$ are isomorphic in a diagonal preserving way if and only the corresponding graph groupoids $\G_E$ and $\G_F$ are isomorphic, and it is proved in \cite{Ste} (see also \cite{ABHS}, \cite{BCH}, and \cite{CR}) that if $R$ is indecomposable and reduced (in particular if $R$ is a integral domain), then there is a diagonal preserving isomorphism between $L_R(E)$ and $L_R(F)$ if and only if $\G_E$ and $\G_F$ are isomorphic. 

By using Corollary \ref{cor}, we show in Theorem \ref{theorem} that two Leavitt path algebras $L_\Z(E)$ and $L_\Z(F)$ of countable graphs are $*$-isomorphic if and only if the groupoids $\G_E$ and $\G_F$ are isomorphic (if and only if there is a diagonal preserving isomorphism between the graph $C^*$-algebras $C^*(E)$ and $C^*(F)$). As is the case with Proposition \ref{pro} and Corollary \ref{cor}, the result in Theorem \ref{theorem} holds for more general subrings of $\C$ than just $\Z$.

The rest of the paper is organised in the following way. In Section \ref{dn} we recall the definitions of directed graphs, the Leavitt path algebras, graph $C^*$-algebras, and graph groupoids; and introduce notation. In Section \ref{sec:result} we present our main result (Theorem \ref{theorem}) and discuss how it is related to \emph{orbit equivalence} of graphs and results in \cite{ABHS}, \cite{AER}, \cite{BCH}, \cite{BCW}, \cite{CR}, \cite{CW}, and \cite{Ste} (Remarks \ref{remark} and \ref{remark2}). We also ask the question if there exist Leavitt path algebras that are isomorphic without the corresponding graph groupoid being isomorphic. If the answer to this question is ``No'', then both \emph{The Isomorphism Conjecture for Graph Algebras} and \emph{The Morita Conjecture for Graph Algebras} introduced in \cite{AT} are true. Finally we present and prove in Section \ref{sec:proof}, Proposition \ref{pro} and Corollary \ref{cor} before we give the proof of Theorem \ref{theorem}. 

\section{Definitions and notation} \label{dn}
We recall in this section the definition of a directed graph, as well as the definitions of the Leavitt path algebra, the graph $C^*$-algebra, and the graph groupoid of a graph; and introduce some notation. Most of this section is copied from \cite{BCW}.

\subsection{Directed graphs}
A \emph{directed graph} is a quadruple $E=(E^0,E^1,s,r)$ where $E^0$ and $E^1$ are sets, and $s$ and $r$ are maps from $E^1$ to $E^0$. A graph $E$ is said to be \emph{countable} if $E^0$ and $E^1$ are countable. 

A {\em path} $\mu$ of length $n$ in $E$ is a sequence of edges $\mu=\mu_1\dots\mu_n$ such that $r(\mu_i)=s(\mu_{i+1})$ for $1\le i\le n-1$. The set of paths of length $n$ is denoted $E^n$. We denote by $|\mu|$ the length of $\mu$. The range and source maps extend naturally to paths: $s(\mu):=s(\mu_1)$ and $r(\mu):=r(\mu_n)$. We regard the elements of $E^0$ as path of length 0, and for $v \in E^0$ we set $s(v):=r(v):=v$. For $v\in E^0$ and $n\in\N_0$ we denote by $vE^n$ the set of paths of length $n$ with source $v$. We define $E^*:=\bigcup_{n\in\N_0}E^n$ to be the collection of all paths with finite length. We define $E^0_\reg:=\{v\in E^0:vE^1\text{ is finite and nonempty}\}$ and $E^0_\sing:=E^0\setminus E^0_\reg$. If $\mu=\mu_1\mu_2\cdots\mu_m, \nu=\nu_1\nu_2\cdots\nu_n\in E^*$ and $r(\mu)=s(\nu)$, then we let $\mu\nu$ denote the path $\mu_1\mu_2\cdots\mu_m\nu_1\nu_2\cdots\nu_n$. 
A {\em loop} (also called a \emph{cycle}) in $E$ is a path $\mu\in E^*$ such that $|\mu|\ge 1$ and $s(\mu)=r(\mu)$. An edge $e$ is an {\em exit} to the loop $\mu$ if there exists $i$ such that $s(e)=s(\mu_i)$ and $e\not=\mu_i$. A graph is said to satisfy \emph{condition (L)} if every loop has an exit. 

An \emph{infinite path} in $E$ is an infinite sequence $x_1x_2\dots $ of edges in $E$ such that $r(e_i)=s(e_{i+1})$ for all $i$. We let $E^\infty$ be the set of all infinite paths in $E$. The source map extends to $E^\infty$ in the obvious way. We let $|x|=\infty$ for $x\in E^\infty$. The \emph{boundary path space} of $E$ is the space 
\begin{equation*}
	\partial E:=E^\infty\cup\{\mu\in E^*:r(\mu)\in E^0_\sing\}.
\end{equation*} 
If $\mu=\mu_1\mu_2\cdots\mu_m\in E^*$, $x=x_1x_2\cdots\in E^\infty$ and $r(\mu)=s(x)$, then we let $\mu x$ denote the infinite path  $\mu_1\mu_2\cdots\mu_m x_1x_2\cdots \in E^\infty$.

For $\mu\in E^*$, the \emph{cylinder set} of $\mu$ is the set
\begin{equation*}
	Z(\mu):=\{\mu x\in\partial E:x\in r(\mu)\partial E\},
\end{equation*} 
where $r(\mu)\partial E:=\{x \in \partial E : r(\mu)=s(x)\}$. Given $\mu\in E^*$ and a finite subset $F\subseteq r(\mu)E^1$ we define
\begin{equation*}
	Z(\mu\setminus F):=Z(\mu)\setminus\left(\bigcup_{e\in F}Z(\mu e)\right).
\end{equation*}

The boundary path space $\partial E$ is a locally compact Hausdorff space with the topology given by the basis $\{Z(\mu\setminus F): \mu\in E^*,\ F\text{ is a finite subset of }r(\mu)E^1\}$, and each such $Z(\mu\setminus F)$ is compact and open (see \cite[Theorem 2.1 and Theorem 2.2]{Web}).

\subsection{Graph $C^*$-algebras}
The {\em graph $C^*$-algebra} of a directed graph $E$ is the universal $C^*$-algebra $C^*(E)$ generated by mutually orthogonal projections $\{p_v:v\in E^0\}$ and partial isometries $\{s_e:e\in E^1\}$ satisfying
\begin{enumerate}
\item[(CK1)] $s_e^*s_e=p_{r(e)}$ for all $e\in E^1$;
\item[(CK2)] $s_es_e^*\le p_{s(e)}$ for all $e\in E^1$;
\item[(CK3)] $\displaystyle{p_v=\sum_{e\in vE^1}s_es_e^*}$ for all $v\in E^0_\reg$.
\end{enumerate}
If $\mu=\mu_1\cdots\mu_n\in E^n$ and $n\ge 2$, then we let $s_\mu:=s_{\mu_1}\cdots s_{\mu_n}$. Likewise, we let $s_v:=p_v$ if $v\in E^0$. Then $\operatorname{span}_\C\{s_\mu s_\nu^*:\mu,\nu\in E^*,\ r(\mu)=r(\nu)\}$ is dense in $C^*(E)$.
We define $\mathcal{D}(E)$ to be the closure in $C^*(E)$ of $\operatorname{span}_\C\{s_\mu s_\mu^*:\mu\in E^*\}$. Then $\mathcal{D}(E)$ is an abelian $C^*$-subalgebra of $C^*(E)$, and it is isomorphic to the $C^*$-algebra $C_0(\partial E)$. We furthermore have that $\mathcal{D}(E)$ is a maximal abelian sub-algebra of $C^*(E)$ if and only if $E$ satisfies condition (L) (see \cite[Example 3.3]{NR}).

\subsection{Leavitt path algebras} \label{lpa}
Let $E$ be a directed graph and $R$ a commutative ring with a unit. The \emph{Leavitt path algebra} of $E$ over $R$ is the universal $R$-algebra $L_R(E)$ generated by pairwise orthogonal idempotents $\{v:v\in E^0\}$ and elements $\{e,e^*:e\in E^1\}$ satisfying
\begin{enumerate}
\item[(LP1)] $e^*f=0$ if $e\ne f$;
\item[(LP2)] $e^*e=r(e)$;
\item[(LP3)] $s(e)e=e=er(e)$;
\item[(LP4)] $e^*s(e)=e^*=r(e)e^*$;
\item[(LP5)] $v=\sum_{e\in vE^1}ee^*$ if $v\in E^0_\reg$.
\end{enumerate}
If $\mu=\mu_1\cdots\mu_n\in E^n$ and $n\ge 2$, then we let $\mu$ be the element $\mu_1\cdots\mu_n\in L_R(E)$ and $\mu^*$ the element $\mu_n^*\cdots\mu_1^*\in L_R(E)$. For $v\in E^0$, we let $v^*:=v$. Then $L_R(E)=\operatorname{span}_R\{\mu \nu^*:\mu,\nu\in E^*,\ r(\mu)=r(\nu)\}$. There is a $\Z$-grading $\bigoplus_{n\in\Z}L_R(E)_n$ of $L_R(E)$ given by $L_R(E)_n=\operatorname{span}_R\{\mu \nu^*:\mu,\nu\in E^*,\ r(\mu)=r(\nu),\ |\mu|-|\nu|=n\}$ (see \cite[Section 3.3]{T}).

We define $D_R(E):=\operatorname{span}_R\{\mu\mu^*:\mu\in E^*\}$. Then $D_R(E)$ is an abelian subalgebra of $L_R(E)$, and it is maximal abelian if and only if $E$ satisfies condition (L) (see \cite[Proposition 3.14 and Theorem 3.22]{CN}). 
If $R$ is a subring of $\C$ that is closed under complex conjugation and contains $1$, then $\mu\nu^*\mapsto\nu\mu^*$ extends to a conjugate linear involution on $L_R(E)$, i.e. $L_R(E)$ is a $*$-algebra. There is an injective $*$-homomorphism $\iota_{L_R(E)}\to C^*(E)$ mapping $v$ to $p_v$ and $e$ to $s_e$ for $v\in E^0$ and $e\in E^1$ (see \cite[Theorem 7.3]{T}).

\subsection{Graph groupoids}
Let $E$ be a directed graph. For $n\in\N_0$, let $\partial E^{\ge n}:=\{x\in\partial E: |x|\ge n\}$. Then $\partial E^{\ge n}= \cup_{\mu \in E^n} Z(\mu)$ is an open subset of $\partial E$. We define the \emph{shift map} on $E$ to be the map $\sigma_E:\partial E^{\ge 1}\to\partial E$ given by $\sigma_E(x_1x_2x_3\cdots)=x_2x_3\cdots$ for $x_1x_2x_3\cdots\in\partial E^{\ge 2}$ and $\sigma_E(e)=r(e)$ for $e\in\partial E\cap E^1$. For $n\ge 1$, we let $\sigma_E^n$ be the $n$-fold composition of $\sigma_E$ with itself. We let $\sigma_E^0$ denote the identity map on $\partial E$. Then $\sigma_E^n$ is a local homeomorphism for all $n\in\N$. When we write $\sigma_E^n(x)$, we implicitly assume that $x\in\partial E^{\ge n}$.  

The \emph{graph groupoid} of a countable directed graph is the locally compact, Hausdorff, \'{e}tale topological groupoid 
\begin{equation*}
	\G_E=\{(x,m-n,y): x,y\in\partial E,\ m,n\in\N_0,\text{ and } \sigma^m(x)=\sigma^n(y)\},
\end{equation*}
with product $(x,k,y)(w,l,z):=(x,k+l,z)$ if $y=w$ and undefined otherwise, and inverse given by $(x,k,y)^{-1}:=(y,-k,x)$. The topology of $\G_E$ is generated by subsets of the form 
\[Z(U,m,n,V):=\{(x,k,y)\in\G_E:x\in U,\ k=m-n,\ y\in V,\ \sigma_E^m(x)=\sigma_E^n(y)\}\] 
where $m,n\in\N_0$, $U$ is an open subset of $\partial E^{\ge m}$ such that the restriction of $\sigma_E^m$ to $U$ is injective, and $V$ is an open subset of $\partial E^{\ge n}$ such that the restriction of $\sigma_E^n$ to $V$ is injective, and $\sigma_E^m(U)=\sigma_E^n(V)$. The map $x\mapsto (x,0,x)$ is a homeomorphism from $\partial E$ to the unit space $\G_E^0$ of $\G_E$. There is a $*$-isomorphism from the $C^*$-algebra of $\G_E$ to $C^*(E)$ that maps $C_0(\G_E^0)$ onto $\mathcal{D}(E)$ (see \cite[Proposition 2.2]{BCW} and \cite[Proposition 4.1]{KPRR}), and a $*$-isomorphism from the Steinberg algebra $A_R(\G_E)$ of $\G_E$ to $L_R(E)$ that maps $\operatorname{span}_R\{1_{Z(Z(\mu),0,0,Z(\mu))}:\mu\in E^*\}$ onto $D_R(E)$ (see \cite[Theorem 2.2]{BCH} and \cite[Example 3.2]{CS}).

\section{The main result}\label{sec:result}

We say that a subring $R$ of $\C$ that is closed under complex conjugation and contains $1$ is \emph{kind} if whenever $\lambda_0,\lambda_1,\dots,\lambda_n\in R$ satisfy $\lambda_0=\sum_{i=0}^n|\lambda_i|^2$, then $\lambda_1=\dots=\lambda_n=0$.

Notice that if a subring $R$ of $\C$ is closed under complex conjugation and contains $1$ and has an essentially unique partition of the unit as defined in \cite{JS}, then it is kind (because if $\lambda_0=\sum_{i=0}^n|\lambda_i|^2$, then $|\lambda_0-1|^2+\sum_{i=1}^n|\lambda_i|^2+\sum_{i=0}^n|\lambda_i|^2=1$). In particular, $\Z$ is kind. The subring of $\C$ generated by $1,\pi^{-1}$ and $\sqrt{1-\pi^{-2}}$ is an example of a kind subring that does not have an essentially unique partition of the unit.

\begin{theorem}\label{theorem}
Let $E$ and $F$ be countable directed graphs, and let $R$ and $S$ be subrings of $\C$ that are closed under complex conjugation and contain $1$, and assume that $R$ is kind. Then the following are equivalent.
\begin{enumerate}
	\item[(1)] The Leavitt path algebras $L_R(E)$ and $L_R(F)$ of $E$ and $F$ are isomorphic as $*$-algebras.
	\item[(2)] There is a $*$-algebra isomorphism $\pi:L_S(E)\to L_S(F)$ such that $\pi(D_S(E))=D_S(F)$.
	\item[(3)] There is a $*$-isomorphism $\phi:C^*(E)\to C^*(F)$ such that $\phi(\D(E))=\D(F)$.
	\item[(4)] The graph groupoids $\mathcal{G}_E$ and $\mathcal{G}_F$ are isomorphic as topological groupoids.
\end{enumerate}
\end{theorem}

Notice that if $R=\Z$, then any $*$-ring isomorphism between $L_R(E)$ and $L_R(F)$ is automatically a $*$-algebra isomorphism. The proof of Theorem \ref{theorem} is given in the next section.

\begin{remark}\label{remark}
Let $\mathcal{P}_E$ and $\mathcal{P}_F$ be the pseudogroups introduced in \cite[Section 3]{BCW}. It is shown in \cite[Theorem 5.3]{AER} that Conditions (3) and (4) each are equivalent to the following two conditions.
\begin{enumerate}
	\item[(5)] There is a homeomorphism $h:\partial E\to\partial F$ that preservs isolated eventually periodic points such that $\{h\circ\alpha\circ h^{-1}: \alpha\in \mathcal{P}_E\}=\mathcal{P}_F$.
	\item[(6)] There is an orbit equivalence $h:\partial E\to\partial F$ as in \cite[Definition 3.1]{BCW} that preservs isolated eventually periodic points.
\end{enumerate}
It follows from \cite[Corollary 4.6]{CW} and the discussion right before \cite[Proposition 3.1]{CW} that if either $E$ and $F$ each satisfy condition (L), or $E$ and $F$ each have only finitely many vertices and no sinks, then any homeomorphism $h:\partial E\to\partial F$ automatically preserves isolated eventually periodic points.
\end{remark}

\begin{remark}\label{remark2}
Suppose $T$ is a commutative ring with a unit and an involution $t\mapsto \overline{t}$ that fixes the unit and the zero element (this could for instance be the identity map). Then there is involution on $L_T(E)$ given by $t\mu\nu^*\mapsto \overline{t}\nu\mu^*$ for $t\in T$ and $\mu,\nu\in E^*$. Thus $L_T(E)$ is a $*$-algebra. 

Condition (4) implies the following condition (see \cite[Theorem 4.1]{CR}). 
\begin{enumerate}
	\item[(7)] There is a $*$-algebra isomorphism $\eta:L_T(E)\to L_T(F)$ such that $\eta(D_T(E))=D_T(F)$.
\end{enumerate}
Obviously, condition (7) implies the following condition.
\begin{enumerate}
	\item[(8)] There is a ring isomorphism $\zeta:L_T(E)\to L_T(F)$ such that $\zeta(D_T(E))=D_T(F)$.
\end{enumerate}
It is shown in \cite[Theorem 6.1]{Ste} (see also \cite[Corollary 4.4]{ABHS}, \cite[Theorem 6.2]{BCH}, and \cite[Corollary 4.1]{CR}) that if $T$ is indecomposable and either $E$ satisfies condition (L), or $T$ is reduced, then (8) implies (4) and (1)--(8) are all equivalent (notice that if $T$ is an integral domain, then it is indecomposable and reduced).
\end{remark}

In the light of Theorem \ref{theorem} and the above remarks, the following question seems natural.

\begin{question}
Do there exist directed graphs $E$ and $F$ and a commutative ring $R$ with a unit such that $L_R(E)$ and $L_R(F)$ are isomorphic as rings, but $\mathcal{G}_E$ and $\mathcal{G}_F$ are not isomorphic?
\end{question}

It follows from Theorem \ref{theorem}, \cite[Theorem 4.2]{CRS}, and Theorem 5 and part 2 of the remarks following Corollary 7 in \cite{AAM} that if the answer to the above question is ``No'', then both \emph{The Isomorphism Conjecture for Graph Algebras} and \emph{The Morita Conjecture for Graph Algebras} introduced in \cite{AT} are true (we cannot rule out the possibility that the conjectures are true even if the answer to the above question is ``Yes'').

\section{Proof of the main result} \label{sec:proof}

Let $E$ be a directed graph and $R$ a subring of $\C$ that is closed under complex conjugation and contains $1$. As in \cite{JS}, we say that $p\in L_R(E)$ is a \emph{projection} if $p=p^*=p^2$.

For the proof of Theorem \ref{theorem} we need the following generalisation of \cite[Theorem 5.6]{JS}.

\begin{proposition}\label{pro}
Let $E$ be a directed graph and let $R$ be a subring of $\C$ that is closed under complex conjugation, contains $1$ and is kind. If $p\in L_R(E)$ is a projection, then $p\in D_R(E)$.
\end{proposition}

\begin{proof}
This proof is inspired by the proof of \cite[Proposition 4.4]{JS}.

For $\mu,\nu\in E^*$, we shall write $\mu\le\nu$ to indicate that there is an $\eta\in E^*$ such that $\mu\eta=\nu$, and $\mu<\nu$ to indicate that $\mu\le\nu$ and $\mu\ne\nu$.

Since $L_R(E)=\operatorname{span}_\Z\{\alpha\beta^*:\alpha,\beta\in E^*\}$, it follows that there are finite subsets $A,B$ of $E^*$ and a family $(\lambda_{(\alpha,\beta)})_{(\alpha,\beta)\in A\times B}$ of elements of $R$ such that
\begin{equation*}
p=\sum_{(\alpha,\beta)\in A\times B}\lambda_{(\alpha,\beta)}\alpha\beta^*.	
\end{equation*}
By repeatedly replacing $\alpha\beta^*$ by $\sum_{e\in r(\alpha)E^1}\alpha ee^*\beta^*$ if necessary, we can assume that there is a $k$ such that $B\subseteq E^k\cup\{\mu\in E^*:|\mu|<k\text{ and }r(\mu)\in E^0_\sing\}$. We can also, by letting some of the $\lambda_{(\alpha,\beta)}$ be 0 if necessary, assume that $B\subseteq A$. Notice that $\alpha\beta^*=0$ unless $r(\alpha)=r(\beta)$. For $\beta\in B$, let $A_\beta:=\{\alpha\in A:r(\alpha)=r(\beta)\}$. We shall also assume that if $\beta\in B$, then there is a least one $\alpha\in A_\beta$ such that $\lambda_{(\alpha,\beta)}\ne 0$ (otherwise we just remove $\beta$ from $B$). We claim that $\lambda_{(\alpha,\beta)}=0$ for all $(\alpha,\beta)\in A\times B$ with $\alpha\in A_\beta\setminus\{\beta\}$, and that $\lambda_{(\beta,\beta)}=(-1)^{m_\beta}$ for all $\beta\in B$ where $m_\beta$ is the number of $\beta'$s in $B$ such that $\beta'<\beta$.

Let $B'=\{\beta\in B:\lambda_{(\alpha,\beta)}=0\text{ for all }\alpha\in A_\beta\setminus\{\beta\}\text{ and }\lambda_{(\beta,\beta)}=(-1)^{m_\beta}\}$, and suppose $B'\ne B$. Choose $\beta\in B\setminus B'$ such that $\beta'<\beta$ for no $\beta'\in B\setminus B'$. Let
\begin{equation*}
F_\beta=\{e\in r(\beta)E^1:\beta e\le\beta'\text{ for some }\beta'\in B\setminus\{\beta\}\}
\end{equation*}
and 
\begin{equation*}
\gamma_\beta=\beta-\beta\sum_{e\in F_\beta}ee^*	
\end{equation*}
($F_\beta=\emptyset$ and $\gamma_\beta=\beta$ unless $|\beta|<k$ and $r(\beta)E^1$ is infinite). Then $\beta'^*\gamma_\beta=0$ for $\beta'\in B$ unless $\beta'\le\beta$.

Since $p=p^*p$, it follows that
\begin{equation*}\label{eq:1}\tag{a}
	\gamma_\beta^*p\gamma_\beta=\gamma_\beta^*p^*p\gamma_\beta.
\end{equation*}
Recall that $L_R(E)$ is $\Z$-graded. The degree 0 part of the left-hand side of \eqref{eq:1} is 
\begin{equation}\label{eq:2}\tag{b}
\sum_{\beta'\in B^{\le\beta}}\lambda_{(\beta',\beta')}\Biggl(r(\beta)-\sum_{e\in F_\beta}ee^*\Biggr)	
\end{equation}
where $B^{\le\beta}:=\{\beta'\in B:\beta'\le\beta\}$, and the degree 0 part of the right-hand side of \eqref{eq:1} is 
\begin{equation}\label{eq:3}\tag{c}
	\begin{split}
		\Biggl(\Biggl(\sum_{\beta'\in B^{<\beta}}\overline{\lambda}_{(\beta',\beta')}\Biggr)&\Biggl(\sum_{\beta'\in B^{<\beta}}\lambda_{(\beta',\beta')}\Biggr)+\sum_{\beta'\in B^{<\beta}}\overline{\lambda}_{(\beta',\beta')}\lambda_{(\beta,\beta)}\\&+\sum_{\beta'\in B^{<\beta}}\lambda_{(\beta',\beta')}\overline{\lambda}_{(\beta,\beta)}+\sum_{\alpha\in A_\beta}|\lambda_{(\alpha,\beta)}|^2\Biggr)\biggl(r(\beta)-\sum_{e\in F_\beta}ee^*\biggr)
	\end{split}
\end{equation}
where $B^{<\beta}:=\{\beta'\in B:\beta'<\beta\}$ (we are using here that $\lambda_{(\alpha,\beta')}=0$ for $\beta'\in B^{<\beta}$ and $\alpha\in A\setminus\{\beta'\}$).

Suppose $m_\beta$ is even. Then $\sum_{\beta'\in B^{<\beta}}\lambda_{(\beta',\beta')}=0$ (because $\lambda_{(\beta',\beta')}=(-1)^{m_{\beta'}}$ for each $\beta'\in B^{<\beta}$). Since $\eqref{eq:2}=\eqref{eq:3}$, it follows that $\lambda_{(\beta,\beta)}=\sum_{\alpha\in A_\beta}|\lambda_{(\alpha,\beta)}|^2$. Since $R$ is kind, it follows that $\lambda_{(\alpha,\beta)}=0$ for $\alpha\in A_\beta\setminus\{\beta\}$ and $\lambda_{(\beta,\beta)}=1$ (recall that $\lambda_{(\alpha,\beta)}\ne 0$ for at least one $\alpha\in A_\beta$), but this contradicts the assumption that $\beta\notin B'$.

If $m_\beta$ is uneven, then $\sum_{\beta'\in B^{<\beta}}\lambda_{(\beta',\beta')}=1$, so it follows from the equality of \eqref{eq:2} and \eqref{eq:3} that $1+\lambda_{(\beta,\beta)}+\overline{\lambda_{(\beta,\beta)}}+\sum_{\alpha\in A_\beta}|\lambda_{(\alpha,\beta)}|^2=1+\lambda_{(\beta,\beta)}$ from which we deduce that $\lambda_{(\alpha,\beta)}=0$ for $\alpha\in A_\beta\setminus\{\beta\}$ and $\lambda_{(\beta,\beta)}=-1$, and thus that $\beta\in B'$. So we also reach a contradiction in this case.

We conclude that we must have that $B'=B$, and thus that $\lambda_{(\alpha,\beta)}=0$ for all $(\alpha,\beta)\in A\times B$ with $\alpha\in A_\beta\setminus\{\beta\}$. Since $\alpha\beta^*=0$ for $\alpha\notin A_\beta$, it follows that $p=\sum_{\beta\in B}\lambda_{(\beta,\beta)}\beta\beta^*\in D_R(E)$.
\end{proof}

\begin{corollary}\label{cor}
Let $E$ and $F$ be directed graphs, and let $R$ be a subring of $\C$ that is closed under complex conjugation, contains $1$ and is kind. If $\pi:L_R(E)\to L_R(F)$ is a $*$-algebra homomorphism, then $\pi(D_R(E))\subseteq D_R(F)$.
\end{corollary}

\begin{proof}
Follows from Proposition \ref{pro} and \cite[Proposition 6.1]{JS}. 
\end{proof}

\begin{proof}[Proof of Theorem \ref{theorem}]
The equivalence of (3) and (4) is proved in \cite{BCW}, and that (4) implies (1) and (2) follows from \cite[Example 3.2]{CS}.

We shall prove $(1)\implies (3)$ and $(2)\implies (3)$.

$(2)\implies (3)$: We shall closely follow the proof of \cite[Lemma 3.5]{JS}. Suppose $\pi:L_S(E)\to L_S(F)$ is a $*$-algebra isomorphism such that $\pi(D_S(E))=D_S(F)$. As in the proof of \cite[Theorem 4.4]{AT}, $\pi$ extends to a $*$-isomorphism $\phi:C^*(E)\to C^*(F)$ satisfying $\phi\circ\iota_{L_S(E)}=\iota_{L_S(F)}\circ\pi$. If $\mu\in E^*$, then 
\begin{equation*}
\phi(s_\mu s_\mu^*)=\phi(\iota_{L_S(E)}(\mu\mu^*))=\iota_{L_S(F)}(\pi(\mu\mu^*))\in \iota_{L_S(F)}(D_S(F))\subseteq\mathcal{D}(F).
\end{equation*} 
Since $\mathcal{D}(E)$ is generated by $\{s_\mu s_\mu^*:\mu\in E^*\}$, it follows that $\phi(\mathcal{D}(E))\subseteq\mathcal{D}(F)$. That $\phi^{-1}(\mathcal{D}(F))\subseteq\mathcal{D}(E)$ follows in a similarly way. Thus $\phi(\mathcal{D}(E))=\mathcal{D}(F)$.

$(1)\implies (3)$: Suppose $\pi:L_R(E)\to L_R(F)$ is a $*$-algebra isomorphism. It follows from Corollary \ref{cor} that $\pi(D_R(E))=D_R(F)$, so an application of the implication $(2)\implies (3)$ with $S=R$ shows that (3) holds.
\end{proof}

\end{document}